\documentclass[12pt,a4paper]{amsart}
\usepackage{a4wide}
\usepackage{amsmath, amssymb, amsfonts,enumerate}
\usepackage[all]{xy}
\usepackage{amscd}
\usepackage{comment}
\usepackage{mathtools}
\usepackage{tikz} 
\usepackage{hyperref}
\usepackage{url}
		
\newtheorem{theorem}{Theorem}[section]
\newtheorem{lemma}[theorem]{Lemma}
\newtheorem{corollary}[theorem]{Corollary}
\newtheorem{proposition}[theorem]{Proposition}

 \theoremstyle{definition}
 \newtheorem{definition}[theorem]{Definition}
 \newtheorem{remark}[theorem]{Remark}

\numberwithin{equation}{section}
\newcommand {\N}{\mathbb{N}} 
\newcommand {\Z}{\mathbb{Z}} 
\newcommand {\R}{\mathbb{R}} 

\newcommand{\WW}{\mathcal{W}}

\DeclareMathOperator{\Per}{Per}

\DeclareMathOperator{\NW}{NW}

\DeclareMathOperator{\End}{End}

\DeclareMathOperator{\Id}{Id}

\begin{document}
\title[A Garden of Eden theorem]{A Garden of Eden theorem for Smale spaces}
\author{Tullio Ceccherini-Silberstein}
\address{Dipartimento di Ingegneria, Universit\`a del Sannio, I-82100 Benevento, Italy}
\address{Istituto Nazionale di Alta Matematica ``Francesco Severi'', I-00185 Rome, Italy}
\email{tullio.cs@sbai.uniroma1.it}
\author{Michel Coornaert}
\address{Universit\'e de Strasbourg, CNRS, IRMA UMR 7501, F-67000 Strasbourg, France}
\email{michel.coornaert@math.unistra.fr}
\begin{abstract}
Given  a dynamical system $(X,f)$ consisting of a  compact metrizable space $X$ and a homeomorphism $f \colon X \to X$,
an endomorphism of $(X,f)$ is a continuous map of $X$ into itself which commutes with $f$.
One says that a dynamical system $(X,f)$ is surjunctive if every injective endomorphism of $(X,f)$ is surjective.
An endomorphism of $(X,f)$ is called pre-injective if its restriction to each $f$-homoclinicity class of $X$ is  injective.
One says that a dynamical system  has the Moore property if every surjective endomorphism of the system is pre-injective
and that it has the Myhill property if every pre-injective endomorphism is surjective.
One says that a dynamical system satisfies the Garden of Eden theorem if it has both the Moore and the Myhill properties.
We prove that every irreducible Smale space satisfies the Garden of Eden theorem
and that every non-wandering Smale space is surjunctive and has the Moore property.
\end{abstract}
\date{\today}
\maketitle

\tableofcontents

\section{Introduction}

Smale spaces were introduced by Ruelle at the end of the  1970s in his influential monograph on thermodynamic formalism~\cite{ruelle-thermodynamic-2nd}.
They form a rich class of uniformly hyperbolic dynamical systems which includes in particular
 all Anosov diffeomorphisms~\cite{anosov-geodesic},  
 all non-wandering sets of Smale's Axiom A diffeomorphisms~\cite{smale-bams},
 all expansive principal algebraic dynamical systems~\cite{schmidt-book}, \cite{li-goe-spec-2019}, \cite{meyerovitch-2019}, 
 all subshifts of finite type~\cite{lind-marcus-second}, 
 as well as various types of solenoids and attractors~\cite{williams_1974}, \cite{wieler_2014}.
Here, by a \emph{dynamical system}, we mean a pair $(X,f)$ consisting of a compact metrizable space $X$, called the \emph{phase space} of the system,  equipped with a homeomorphism $f \colon X \to X$.
Roughly speaking, a dynamical system $(X,f)$ is a \emph{Smale space} if $X$ admits a local product structure such that  $f$ is uniformly contracting in one direction and uniformly expanding in the other one
(see Section~\ref{sec:background} for precise definitions). 
An \emph{endomorphism} of a dynamical system $(X,f)$ is a continuous map $\tau \colon X \to X$ which commutes with $f$, i.e., such that $f \circ \tau = \tau \circ f$.
One says that a dynamical system $(X,f)$ is \emph{surjunctive} if every injective endomorphism of $(X,f)$ is surjective.
As a trivial example, any dynamical system with finite phase space is surjunctive since every injective map of a finite set into itself is surjective. 
On the other hand,   the dynamical system $(X,f)$, where $X$ is the real interval $[0,1]$ and $f \colon X \to X$ is the identity map, is not surjunctive since the map $\tau \colon X \to X$, defined by $\tau(x) \coloneqq x/2$ for all $x \in X$, is an injective endomorphism of $(X,f)$ which is not surjective.
The word ``surjunctive" was created by Gottschalk~\cite{gottschalk}. 
\par
Given a dynamical system $(X,f)$ and a compatible metric on $X$, one says that two points $x,y \in X$ are \emph{homoclinic} if the distance from $f^n(x)$ to $f^n(y)$ converges to zero as $|n|$ tends to infinity.
This definition does not depend on the choice of the compatible metric by compactness of $X$.
Homoclinicity is an equivalence relation on the phase space.
One says that an endomorphism of a dynamical system  is \emph{pre-injective} if its restriction to every homoclinicity class of the phase space is injective.
One says that a dynamical system $(X,f)$ has the \emph{Moore property} if every surjective endomorphism of $(X,f)$ is pre-injective
and that $(X,f)$ has the \emph{Myhill property} if every pre-injective endomorphism of $(X,f)$ is surjective.
Note that injectivity implies pre-injectivity.
Therefore, every  dynamical system which has the Myhill property is surjunctive.
One says that a dynamical system has the Moore-Myhill property, or that it satisfies the \emph{Garden of Eden theorem}, if it satisfies
both the Moore and the Myhill properties.
 The Garden of Eden theorem was first established for full shifts by Moore~\cite{moore} and Myhill~\cite{myhill}.
It was subsequently extended  to all irreducible subshifts of finite type
by Fiorenzi in~\cite{fiorenzi-sofic}.
She also showed that irreducible sofic subshifts have the Myhill property but may fail to have the Moore property.
In~\cite[Theorem~1.1]{csc-anosov-2016}, it is shown that Anosov diffeomorphisms on tori  satisfy the Garden of Eden theorem.
\par
  In the last decades, several results on surjunctivity and Garden of Eden type theorems have been established  in many other categories of dynamical systems
(continuous actions of groups other than $\Z$, additional structure on the phase space, etc.). 
See the surveys \cite{csc-cat}, \cite{csc-goe-old-and-new-2020}, the monographs 
\cite{csc-cag2}, \cite{csc-ecag}, and the references therein.
\par
In the present paper, we shall prove the following.

\begin{theorem}
\label{t:goe-for-irred-smale}
Every irreducible Smale space satisfies the Garden of Eden theorem.
\end{theorem}

The result of Fiorenzi mentioned above, namely  that irreducible subshifts of finite type satisfy the Garden of Eden theorem, is covered
by Theorem~\ref{t:goe-for-irred-smale} since all subshifts of finite type are Smale spaces.
Theorem~\ref{t:goe-for-irred-smale} also shows that every irreducible Anosov diffeomorphism satisfies the Garden of Eden theorem.
As all Anosov diffeomorphisms on tori are known to be irreducible, Theorem~\ref{t:goe-for-irred-smale} also covers
Theorem~1.1 in~\cite{csc-anosov-2016}.
  However, it is an open question whether every Anosov diffeomorphism  is irreducible. 
For an Anosov diffeomorphism $f \colon M \to M$,    the following conditions are known to be equivalent 
(see e.g.~\cite[Theorem 5.10.3]{brin-stuck}):
(a) $(M,f)$ is irreducible;
(b) $(M,f)$ is topologically mixing;
(c) $(M,f)$ is non-wandering; 
(d) the periodic points of $f$ are dense in $M$.
All known examples of compact manifolds admitting Anosov diffeomorphisms are homeomorphic to \emph{infra-nilmanifolds}, i.e.,
compact quotients of simply-connected nilpotent Lie groups by discrete groups of isometries.
Manning~\cite{manning-1974} proved that all Anosov diffeomorphisms of infra-nilmanifolds are irreducible.
Thus, Theorem~\ref{t:goe-for-irred-smale} gives the following.

\begin{corollary}
\label{c:anosov-diff-infranil}
Every Anosov diffeomorphism of an infra-nilmanifold satisfies the Garden of Eden theorem.
\end{corollary}

There exist Smale spaces that are not surjunctive (see Remark~\ref{r:non-surjunctive-sft})
as well as Smale spaces that do not have the Myhill property nor the Moore property (see Remark~\ref{r:smale-space-not-moore}).
On the other hand, there exist non-wandering Smale spaces that do not have the Myhill property (Remark~\ref{r:smale-space-not-myhill}).
However, we shall establish the following.

\begin{theorem}
\label{t:main-surjunctive}
Every non-wandering Smale space is surjunctive and has the Moore property.
\end{theorem}

As every subshift of finite type is a Smale space, Theorem~\ref{t:main-surjunctive} gives us the following.

\begin{corollary}
\label{c:nw-sft-is surjunc}
Let $A$ be a finite set. 
Then every non-wandering subshift of finite type  $\Sigma \subset A^{\Z}$ 
is surjunctive and has the Moore property.
\end{corollary}

The paper is organized as follows.
In Section~\ref{sec:background}, we introduce notation and collect background material.
The proof of the first part of Theorem~\ref{t:main-surjunctive} is given in Section~\ref{sec:surjunctivity}.
The proof of Theorem~\ref{t:goe-for-irred-smale} is given in Section~\ref{sec:goe-irred-smale}.
The proof of the second part of Theorem~\ref{t:main-surjunctive} is given in Section~\ref{sec:nw-surj-moore}.

\section{Preliminaries}
\label{sec:background}
In this section we fix notation  and collect some background material. 

\subsection{General notation}
We write $\Z$ for the set of integers and $\N$ for the set of non-negative integers.
\par
Given a set $X$, we write $|X|$ for the cardinality of $X$ and we denote by $\Id_X$ the identity map on $X$.
We use the symbol $\circ$ for the composition of maps (with the convention $(g \circ f)(x) \coloneqq g(f(x))$
for $f \colon X \to Y$, $g \colon Y \to Z$, and $x \in X$).
Given a map $f \colon X \to X$, we set $f^0 \coloneqq \Id_X$ and define inductively $f^n$ for $n \in \N$ by setting $f^{n + 1} \coloneqq f \circ f^n$. 
In the case when $f$ is bijective, we write $f^{-1}$ for the inverse of $f$ and 
$f^{-n} \coloneqq (f^{-1})^n$ for every $n \in \N$.
\par
If $Y$ is a subset of a topological space $X$, we write $\overline{Y}$ for the closure of $Y$ in $X$.
\par
If $X$ is a metrizable space, we denote by $d_X$, or simply $d$ if there is no risk of confusion, any compatible metric on $X$.

\subsection{Dynamical systems}
(See~\cite{brin-stuck}, \cite{ruelle-thermodynamic-2nd}, \cite{nekrashevych})
\begin{definition}[Dynamical system]
A \emph{dynamical system} is a pair $(X,f)$, where
$X$ is a compact metrizable space  and  
 $f \colon X \to X$ is  a homeomorphism.
\end{definition}
 
 Let $(X,f)$ be a dynamical system.
 A subset $Y \subset X$ is said to be \emph{invariant}   if $f(Y) = Y$.
If $Y \subset X$ is an invariant subset, we denote by $f\vert_Y$ the restriction of $f$ to 
$Y$, i.e., the homeomorphism
$f\vert_Y \colon Y \to Y$ given by $f\vert_Y(y)  := f(y)$ for all $y \in Y$.
When $Y \subset X$ is a closed invariant subset, then the pair $(Y,f|_Y)$ is also a dynamical system
and one says that $(Y,f|_Y)$ is a \emph{subsystem} of $(X,f)$.
We shall write it  $(Y,f)$, or even simply $Y$,  if there is no risk of confusion.
\par
A point $x \in X$ is said to be \emph{periodic} if there exists an integer $n \geq 1$ such that $f^n(x) = x$.
One then says that the integer $n$ is a \emph{period} of $x$.
We denote by $\Per_n(X,f)$ the subset of $X$ consisting  of all periodic points of $(X,f)$ admitting $n$ as a period.
Note that $\Per_n(X,f)$ is a closed invariant subset of $X$.
The set $\Per(X,f) \coloneqq \bigcup_{n \geq 1} \Per_n(X,f)$  of all periodic points of $f$ is also an invariant subset of $X$.
\par
A point $x \in X$ is said to be \emph{non-wandering} if for every neighborhood $U$ of $x$, there exists an integer $n \geq 1$ such that $U \cap f^n(U) \not= \varnothing$.
The set $\NW(X,f)$ consisting of all non-wandering points of $f$ is a closed invariant subset of $X$.
Observe that $\Per(X,f) \subset \NW(X,f)$.
One says that the dynamical system $(X,f)$ is \emph{non-wandering} if $\NW(X,f) = X$.
\par
One says that $(X,f)$ is \emph{topologically mixing} if for all non-empty subsets $U$ and $V$ of $X$, there exists an integer  $N \geq 1$ such that $U \cap f^n(V) \not= \varnothing$ for all  $n \geq N$.
One says that $(X,f)$ is \emph{irreducible}, or \emph{topologically $+$-transitive},  if for all non-empty open subsets $U$ and $V$ of $X$, there exists $n \geq 1$ such that $U \cap f^n(V) \not= \varnothing$.
The dynamical system $(X,f)$ is irreducible if and only if there exists a point $x \in X$ whose forward orbit
$\{f^n(x) : n \geq 1\}$ is dense in $X$.
We have the following   implications:
  \[
  \text{topologically mixing } \implies
  \text{irreducible } \implies
  \text{non-wandering}.
  \]
Each of these implications is strict.
Indeed, consider a discrete space $X$ consisting of two distinct points and let $f \colon X \to X$ be the homeomorphism of $X$ which exchanges the two  points of $X$.
Then the dynamical system $(X,f)$ is irreducible but not topologically mixing.
On the other hand, the dynamical system $(X,\Id_X)$ is non-wandering but not irreducible.
\par
Let $X$ be a compact metrizable space
equipped with a homeomorphism $f \colon X \to X$. 
One says that the dynamical system $(X,f)$ has the \emph{weak specification} property 
if
for every $\varepsilon > 0$ 
there exists an integer $N \geq 1$ such that
for any  sequence $I_1,I_2,\dots,I_k$ of finite intervals of $\Z$
satisfying  $\min(I_{i+1}) -  \max(I_i) \geq N$ for each $i \in \{1,2,\dots,k-1\}$
and any sequence $x_1,x_2,\dots,x_k$ of points of $X$,
there exists a point $x \in X$ such that
$d(f^n(x),f^n(x_i)) \leq \varepsilon$ for all $i \in \{1,2,\dots,k\}$ and $n \in I_i$.
By compactness of $X$, this definition does not depend on the choice of the compatible metric $d$. 
\par
   Let $(X,f)$ and $(Y,g)$ be two dynamical systems.
 A continuous map $\varphi \colon X \to Y$ is called a \emph{morphism of dynamical systems}  if it satisfies
$\varphi \circ f = g \circ \varphi$.
Note that $\varphi(X)$ is then a closed invariant subset of $(Y,g)$.
The class of all dynamical systems together with the morphisms between them form a category.

\begin{proposition}
\label{p:prop-morphism}
Let  $\varphi \colon (X,f) \to (Y,g)$ be a morphism of dynamical systems.
Then the following hold:
\begin{enumerate}[\rm (i)]
\item
$\varphi(\Per_n(X,f)) \subset \Per_n(Y,g)$ for every   $n \geq 1$;
\item
$\varphi(\Per(X,f)) \subset \Per(Y,g)$;
\item
$\varphi(\NW(X,f)) \subset \NW(Y,g)$;
\item
if $(X,f)$ is non-wandering 
(resp.~is irreducible, resp.~is topologically mixing, resp.~has the weak specification property) 
then $(\varphi(X),g)$ is non-wandering
(resp.~is irreducible, resp.~is topologically mixing, resp.~has the weak specification property). 
\end{enumerate}
\end{proposition}

\begin{proof}
(i) If $x \in \Per_n(X,f)$, then $g^n(\varphi(x)) = \varphi(f^n(x))
= \varphi(x)$
so that $\varphi(x) \in \Per_n(Y,g)$. 
\par
(ii) Using (i), we get $\varphi(\Per(X,f)) = \varphi(\bigcup_{n \geq 1} \Per_n(X,f)) = \bigcup_{n \geq 1} \varphi(\Per_n(X,f)) \subset \bigcup_{n \geq 1} \Per_n(Y,g) = \Per(Y,g)$.
\par
(iii) Let $x \in \NW(X,f)$ and let $y \coloneqq \varphi(x)$.
Let $V$ be a neighborhood of $y$ in $\varphi(X)$.
Then $U \coloneqq \varphi^{-1}(V)$ is a neighborhood of $x$.
Since $x \in \NW(X,f)$, there exists $n \geq 1$ such that $V \cap f^n(V) \not= \varnothing$.
As $\varphi(V \cap f^n(V)) \subset \varphi(V) \cap \varphi(f^n(V)) = U \cap g^n(\varphi(V)) \subset U \cap g^n(U)$,
this implies $U \cap g^n(U) \not= \varnothing$. Therefore $y \in \NW(Y,g)$.
\par
(iv)
The fact that if $(X,f)$ is non-wandering then $(\varphi(X),g)$ is non-wandering  immediately follows from (iii).
\par
Observe that if $U$ and $V$ are any two non-empty open subsets of $\varphi(X)$,
then $U' \coloneqq \varphi^{-1}(U)$ and $V' \coloneqq \varphi^{-1}(V)$
are two non-empty open subsets of $X$.
As $U' \cap f^n(V') \not= \varnothing$ implies that  $\varnothing \not= \varphi(U' \cap f^n(V')) \subset \varphi(U') \cap \varphi(f^n(V')) = \varphi(U') \cap g^n(\varphi(V')) \subset U \cap g^n(V)$,
we deduce that if $(X,f)$ is irreducible (resp.~topologically mixing) then $(\varphi(X),g)$ is irreducible (resp.~topologically mixing).
\par
Finally, suppose that $(X,f)$ has the weak specification property
and let $\varepsilon > 0$.
Since $\varphi$ is uniformly continuous, there exists $\delta > 0$ such that
$d_X(x_1,x_2) \leq \delta$ implies $d_Y(\varphi(x_1),\varphi(x_2)) \leq \varepsilon$ for all $x_1,x_2 \in X$.
Suppose we are given
a finite sequence $I_1,I_2,\dots,I_k$ of finite intervals of $\Z$
and a sequence $y_1,y_2,\dots,y_k$ of points of $\varphi(X)$.
Choose $x_1,x_2, \dots, x_k \in X$ such that $\varphi(x_i) = y_i$ for each $i \in \{1,2,\dots,k\}$.
Since $(X,f)$ has the weak specification property,
there exists an integer $N \geq 1$ such that 
if $\min(I_{i+1}) -  \max(I_i) \geq N$ for each $i \in \{1,2,\dots,k-1\}$,
then there exists $x \in X$ such that
$d_X(f^n(x),f^n(x_i)) \leq \delta$ for all $i \in \{1,2,\dots,k\}$ and $n \in I_i$.
 This implies that  the point $y \coloneqq \varphi(x) \in \varphi(X)$ satisfies
$d_Y(g^n(y),g^n(y_i)) = d_Y(g^n(\varphi(x)),g^n(\varphi(x_i)))
= d_Y(\varphi(f^n(x)),\varphi(f^n(x_i))) \leq \varepsilon$ for all $i \in \{1,2,\dots,k\}$ and $n \in I_i$.
Thus,  $(\varphi(X),g)$ has the weak specification property.
This completes the proof of (iv).
\end{proof}

Two dynamical systems are said to be \emph{topologically conjugate} if they are isomorphic objects in the category of dynamical systems.
An isomorphism between two dynamical systems is also called a \emph{topological conjugacy}.
A surjective morphism $\pi \colon X \to Y$ is called a \emph{factor map}.
One says that $(X,f)$ is an \emph{extension} of $(Y,g)$ and that $(Y,g)$ is a \emph{factor} of $(X,f)$ if there exists a factor map $\pi \colon X \to Y$.
It follows from Assertion~(iv) in  Proposition ~\ref{p:prop-morphism}
 that every factor of a dynamical system which is 
 non-wandering (resp.~irreducible, resp.~topologically mixing, resp.~with the weak specification property) 
 is itself
 non-wandering (resp.~irreducible, resp.~topologically mixing, resp.~with the weak specification property).
\par
An \emph{endomorphism} of the dynamical system $(X,f)$ is a continuous map $\tau \colon X \to X$ such that $\tau \circ f = f \circ \tau$.
We write $\End(X,f)$ for the set of all endomorphisms of $(X,f)$.
Observe that $f^n \in \End(X,f)$ for all $n \in \Z$ and that $\End(X,f)$ is a monoid for the composition of maps.  
\par
\subsection{Stable equivalence, unstable equivalence, and homoclinicity}
Let $(X,f)$ be a dynamical system and let $d$ be a compatible metric on $X$.
Two points $x, y \in X$ are called \emph{stably equivalent} (resp.~\emph{unstably equivalent}, resp.~\emph{homoclinic}),
and one writes $x \sim_s y$ (resp.~$x \sim_u y$, resp.~$x \sim_h y$), 
if  $d(f^n(x),f^n(y)) \to 0$ as $n \to \infty$ (resp.~$n \to - \infty$, resp.~$|n| \to \infty$).
Clearly, $\sim_s$ (resp.~$\sim_u$, resp.~$\sim_h$)  is an equivalence relation on $X$.
By compactness of $X$, 
each of these equivalence relations  does not depend on the choice of the compatible metric $d$.
Given $x \in X$,
we denote by  $W^s(X,f,x)$ (resp.~$W^u(X,f,x)$, resp.~$W^h(X,f,x)$), or simply
 $W^s(x)$ (resp.~$W^u(x)$, resp.~$W^h(x)$),
 the stable (resp.~unstable, resp.~homoclinic) equivalence class of $x$.
 Observe that $W^h(x) = W^s(x) \cap W^u(x)$ and that $f^n(W^s(x)) = W^s(f^n(x))$ (resp.~$f^n(W^u(x)) = W^u(f^n(x))$, resp.~$f^n(W^h(x)) = W^h(f^n(x))$) 
 for all $x \in X$ and $n \in \Z$.

\begin{proposition}
\label{p:prop-morphism-stable}
Let  $\varphi \colon (X,f) \to (Y,g)$ be a morphism of dynamical systems
and let $x \in X$.
Then the following hold:
\begin{enumerate}[\rm (i)]
\item
$\varphi(W^s(X,f,x)) \subset W^s(Y,g,\varphi(x))$;
\item
$\varphi(W^u(X,f,x)) \subset W^u(Y,g,\varphi(x))$;
\item
$\varphi(W^h(X,f,x)) \subset W^h(Y,g,\varphi(x))$.
\end{enumerate}
\end{proposition}

\begin{proof}
Let $d_X$ (resp.~$d_Y$) be a compatible metric on $X$ (resp.~$ Y$).
Take $x' \in W^s(X,f,x)$.
This means that  $d_X(f^n(x'),f^n(x)) \to 0$ as $n \to \infty$.
By continuity of $\varphi$, we have
\[
d_Y(\varphi(f^n(x')),\varphi(f^n(x))) \to 0
\] 
as $n \to \infty$.
As $\varphi \circ f = g \circ \varphi$, we deduce that
$d_Y(g^n(\varphi(x')),g^n(\varphi(x))) \to 0$ as $n \to \infty$.
 This shows that $\varphi(x') \in W^s(Y,g,\varphi(x))$ and hence that
 $\varphi(W^s(X,f,x)) \subset W^s(Y,g,\varphi(x))$.
 The proofs of the two other inclusions are similar.
 It suffices to replace $n \to \infty$ by $n \to - \infty$ (resp.~$|n| \to \infty$).
\end{proof}

\begin{proposition}
\label{p:homoclinic-f-f-k}
Let $X$ be a compact metrizable space and let $f \colon X \to X$ be a homeomorphism.
Let $k \geq 1$ be an integer and let $x, y \in X$.
Then the following conditions are equivalent.
\begin{enumerate}[\rm (i)]
\item
$x$ and $y$ are homoclinic points of  $(X,f)$;
\item
$x$ and $y$ are homoclinic points of  $(X,f^k)$.
\end{enumerate}
\end{proposition}

\begin{proof}
Suppose first that  $x$ and $y$ are homoclinic points of $(X,f)$. 
This means that $d(f^n(x),f^n(y)) \to 0$ as $|n| \to \infty$.
As this implies $d((f^k)^n(x),(f^k)^n(y)) = d(f^{kn}(x),f^{kn}(y)) \to 0$ as $|n| \to \infty$,
we deduce that $x$ and $y$ are homoclinic for $(X,f^k)$.
\par
Conversely, suppose now that  $x$ and $y$ are homoclinic points of $(X,f^k)$.
Let $\varepsilon > 0$.
By compactness of $X$,
the maps $f, \dots, f^k$ are uniform homeomorphisms.
Consequently, there exists $\delta > 0$ such that $d(x_1,x_2) \leq \delta$ implies $d(f^i(x_1),f^i(x_2)) \leq \varepsilon$ for all $x_1,x_2 \in X$ and $i \in \{1,\dots,k\}$.
Since $x$ and $y$ are homoclinic for $(X,f^k)$, there exists $N \geq 1$ such that
$d(f^{kn}(x),f^{kn}(y)) = d((f^k)^n(x),(f^k)^n(y))   \leq \delta$ for all $n$ such that $|n| \geq N$.
By our choice of $\delta$, this implies
$d(f^{kn + i}(x),f^{kn+i}(y)) = d(f^i(f^{kn}(x)),f^i(f^{kn}(y)) \leq \varepsilon$ for all $n,i$ such that $|n| \geq N$ and $i \in \{1,\dots,k\}$.
We deduce that   
 $d(f^{n}(x),f^{n}(y)) \leq \varepsilon$ for all $n$ such that $|n| \geq kN + 1$.
 As $\varepsilon > 0$ was arbitrary, this shows that the points $x$ and $y$ are homoclinic for $(X,f)$.
\end{proof}

One says that a morphism $\varphi \colon (X,f) \to (Y,g)$ of dynamical systems is \emph{pre-injective} if the restriction of $\varphi$ to every homoclinicity class of $X$ is injective.

\begin{proposition}
\label{p:post-compose-pre-inj}
Let $\varphi \colon (X,f) \to (Y,g)$ and $\psi \colon (Y,g) \to (Z,h)$ be morphisms of dynamical systems.
Suppose that $\psi \circ \varphi \colon (X,f) \to (Z,h)$ is pre-injective.
Then $\varphi$ is pre-injective.
\end{proposition}

\begin{proof}
Let $x_1$ and $x_2$ be homoclinic points in $X$ and suppose that $\varphi(x_1) = \varphi(x_2)$.
We then have $(\psi \circ \varphi)(x_1) = \psi(\varphi(x_1)) = \psi(\varphi(x_2)) = (\psi \circ \varphi)(x_2)$.
As $\psi \circ \varphi$ is pre-injective, this implies that $x_1 = x_2$.
This shows that $\varphi$ is pre-injective.
\end{proof}

\begin{corollary}
\label{c:pre-injective-powers}
Let $(X,f)$ be a dynamical system and let $\tau \in \End(X,f)$.
Suppose that there exists an integer $k \geq 1$ such that $\tau^k$ is pre-injective.
Then $\tau$ is pre-injective. 
\end{corollary}

\begin{proof}
This immediately follows from Proposition~\ref{p:post-compose-pre-inj} after writing $\tau^k = \tau^{k - 1} \circ \tau$.
\end{proof}

\subsection{Expansivity}
Let $X$ be a compact metrizable space and
let $d$ be a compatible metric on $X$.
A homeomorphism $f \colon X \to X$ is called \emph{expansive} if there exists a constant $\eta > 0$ such that if $x,y \in X$ and $x \not= y$ then there exists $n \in \Z$ such that $d(f^n(x),f^n(y)) > \eta$.
One then says that $\eta$ is an \emph{expansivity constant} for $(X,f,d)$.
By compactness, the definition of expansivity for a homeomorphism of $X$ does not depend on the choice of the metric $d$.
Note that $\eta$ is an expansivity constant for $(X,f,d)$ if and only if it satisfies the following condition:
for all $x,y \in X$, if one has $d(f^n(x),f^n(y)) \leq \eta$ for all $n \in \Z$ then $x = y$.

\begin{lemma}
\label{l:exp-finitely-p-per}
local expansivity sauvegardeLet $X$ be a compact metrizable space and let $f \colon X \to X$ be an expansive homeomorphism.
Then the set $\Per_n(X,f)$ is finite for every integer $n \geq 1$.
\end{lemma}

\begin{proof}
Choose a compatible metric $d$ on $X$ and let $\eta > 0$ be an expansivity constant for $(X,f,d)$.
Let $n \geq 1$ be an integer.
As the maps $f,f^2,\dots,f^n$ are uniformly continuous, there exists $\varepsilon > 0$ such that
$d(f^i(x),f^i(y)) \leq \eta$ for all $1 \leq i \leq n$ and $x,y \in X$ with $d(x,y) \leq \varepsilon$.
Suppose by contradiction that $\Per_n(X,f)$ is infinite.
Then, by compactness of $X$, there are distinct points $x,y \in \Per_n(X,f)$ such that $d(x,y) \leq \varepsilon$.
We then have $d(f^i(x),f^i(y)) \leq \eta$ for all $i \in \Z$, a contradiction since $\eta$ is an expansivity constant for $(X,f,d)$.
This shows that $\Per_n(X,f)$ is finite.
\end{proof}

\subsection{Subshifts} (See~\cite{lind-marcus-second}, \cite[Chapter~1]{csc-ecag})
Let $A$ be a finite discrete space.
We equip the set $A^{\Z}$ of all bi-infinite sequences $u = (u_n)_{n \in \Z}$, where $u_n \in A$ for all $n \in \Z$,   with the topology of pointwise convergence.
If we regard $A^{\Z}$ as the product of a family of copies of $A$ indexed by $\Z$,
the topology on $A^{\Z}$ is the product topology.
The space $A^{\Z}$ is compact, totally disconnected, and metrizable.
A compatible metric $d$ on $A^{\Z}$ is defined by
setting, for all $u,v \in A^{\Z}$,
\[
d(u,v) \coloneqq
\begin{cases}
0 &\text{ if }u = v, \\ 
2^{-\inf\{n \in \N : (u_{-n},u_n) \not= (v_{-n},v_n)\}} &\text{ if }u \not= v.
\end{cases}
\]
\par
The \emph{full shift} over $A$ is the dynamical system $(A^{\Z},\sigma)$, where $\sigma \colon A^{\Z} \to A^{\Z}$ is the homeomorphism defined, for all $u \in A^{\Z}$,  by
$\sigma(u) \coloneqq v$ with $v_n \coloneqq u_{n - 1}$ for all $n \in \Z$.
Note that $(A^{\Z},\sigma)$ is expansive since $1/2$ is clearly an expansivity constant for $(A^{\Z},\sigma,d)$.
\par
Two sequences $u,v \in A^{\Z}$ are $\sigma$-homoclinic if and only if there exists a finite subset $K \subset \Z$ such that $u_n = v_n$ for all $n \in \Z \setminus K$ \cite[Exercise~5.2.(a)]{csc-ecag}.
\par
A closed $\sigma$-invariant subset $\Sigma \subset A^{\Z}$ is called a \emph{subshift}.
 \par
 Let $A^*$ denote the free monoid based on $A$.
Every element of  $A^*$ can be uniquely written as a \emph{word} $w = a_1a_2\cdots a_n$, where $n \geq 0$ is the \emph{length} of the word $w$, and $a_i \in A$ for all $1 \leq i \leq n$.
The monoid operation on $A^*$ is the concatenation of words.
The \emph{language} of a subshift $\Sigma \subset A^{\Z}$ is the subset $L(\Sigma) \subset A^*$ consisting of all words $w \in A^*$ such that there exist $n \geq 0$ and $u \in \Sigma$ such that
$w = u_1u_2\cdots u_n$.
\par
If $\Sigma \subset A^*$ is a subshift,
one shows that the dynamical system $(\Sigma,\sigma)$ is irreducible  if and only if, for all $u,v \in L(\Sigma)$, there exists $w \in A^*$ such that $u w v \in L(\Sigma)$
\cite[Exercise~1.76]{csc-ecag}.
\par  
A subshift $\Sigma \subset A^{\Z}$ is called a \emph{subshift of finite type} if there exist an integer $m \geq 0$ and a subset $F \subset A^*$  consisting of words of length $m + 1$ such that
$\Sigma$ is the set of all sequences $u = (u_n)_{n \in \Z} \in A^{\Z}$
satisfying $u_n u_{n+1} \cdots u_{n + m} \in F$ for all $n \in \Z$.
One then says that $\Sigma$ is an $m$-step subshift of finite type and that $F$ is a \emph{defining set of admissible words} for $\Sigma$.
Observe that the full shift $A^{\Z}$ is a $0$-step subshift of finite type admitting $A$ as a defining set of admissible words.
 \par
 One says that a subshift $\Sigma \subset A^{\Z}$ is \emph{sofic} if there exists a finite set $B$ such that $\Sigma$ is a factor of some subshift of finite type $\Sigma' \subset B^{\Z}$.
\par
A subshift $\Sigma \subset A^{\Z}$ is said to be \emph{strongly irreducible} if there exists an integer $N \geq 1$ such that the following holds: 
if $I_1$ and $I_2$ are finite intervals of $\Z$ such that $\min(I_2) - \max(I_1) \geq N$ and $x_1, x_2 \in \Sigma$, then there exists $x \in \Sigma$ such that $x\vert_{I_1} = x_1\vert_{I_1}$ and $x\vert_{I_2} = x_2\vert_{I_2}$. 
It can be shown that a subshift $\Sigma \subset A^{\Z}$ has the weak specification property if and only if $\Sigma$ is strongly irreducible
(cf.~\cite[Proposition~A.1]{li-goe-spec-2019}, \cite[Exercise~1.69]{csc-ecag}). 
Every  topologically mixing subshift  of finite type is strongly irreducible 
(cf.~\cite[Exercise~1.80]{csc-ecag}). 
Consequently, we have the following.

\begin{proposition}
\label{p:tmftwsp}
Let $A$ be a finite set and let $\Sigma \subset A^{\Z}$ be a topologically mixing subshift of finite type.
Then $\Sigma$ has the weak specification property.
\end{proposition}

\subsection{The pseudo-orbit tracing property}
Let $(X,f)$ be a dynamical system and let $d$ be a compatible metric on $X$.
\par
Given $\varepsilon > 0$, one says that a sequence $(u_n)_{n \in \Z}$ of points of $X$ is $\varepsilon$-\emph{traced} by the orbit of a point $x \in X$ if one has $d(u_n,f^n(x)) \leq \varepsilon$ for all $n \in \Z$.
\par
Given $\delta > 0$, one says that a sequence $(u_n)_{n \in \Z}$ of points of $X$ is a $\delta$-\emph{pseudo-orbit} if one has
$d(u_{n+1},f(u_n)) \leq \delta$ for all $n \in \Z$.
\par
One says that the dynamical system $(X,f)$ has the \emph{pseudo-orbit tracing property} if for every $\varepsilon > 0$, there exists $\delta > 0$ such that every $\delta$-pseudo-orbit in $X$ is $\varepsilon$-traced by the orbit of some point in $X$.
By compactness of $X$,  the fact that $f$ has the pseudo-orbit tracing property or not does not depend on the choice of the compatible metric $d$.

 \subsection{Smale spaces}  
We first recall  Ruelle's original definition of a Smale space (cf.~\cite[Section~7.1]{ruelle-thermodynamic-2nd}, \cite[Definition~2.1.6]{putnam-homology-smale-spaces-2014}
\cite[Section~2.2]{putnam-notes-smale-spaces-2015}).
Let $X$ be a compact metrizable space equipped with a homeomorphism $f \colon X \to X$.
One says that $(X,f)$ is a \emph{Smale space} if there exist a compatible metric $d$ on $X$,  constants $\varepsilon > 0$ and $0 < \lambda < 1$,
and a continuous map 
\[
 [\cdot,\cdot] \colon \Delta_{\varepsilon} \coloneqq \{(x,y) \in X \times X : d(x,y) \leq \varepsilon\} \to X
\]
satisfying the following conditions for all $x,y,z \in X$:
\begin{enumerate}[\rm (Sm1)]
\item
$[x,x] = x$ for all $x \in X$;
\item
$[[x, y], z] = [x, z]$ whenever both sides are defined;
\item
$[x, [y, z]] = [x, z]$ whenever both sides are defined;
\item
$[f(x),f(y)] = f([x,y])$ whenever both sides are defined;
\item
$d(f(y),f(z)) \leq \lambda d(y,z)$ whenever $(y,x),(z,x) \in \Delta_\varepsilon$ and $[y,x] = x= [z,x]$;
\item
$d(f^{-1}(y),f^{-1}(z)) \leq \lambda d(y,z)$ whenever $(x,y),(x,z) \in \Delta_{\varepsilon}$ and $[x,y] = x = [x,z]$.
\end{enumerate}

It is known that every Smale space is expansive \cite[Section~7.3]{ruelle-thermodynamic-2nd}, \cite[Proposition~2.1.9]{putnam-homology-smale-spaces-2014}
and has the pseudo-orbit tracing property 
\cite[Section~7.3]{ruelle-thermodynamic-2nd}.
Conversely, if $(X,f)$ is a dynamical system such that $f$ is expansive and has the pseudo-orbit tracing property, then $(X,f)$ is a Smale space \cite[Theorem]{ombach-smale-spaces}.
Thus, an equivalent but more concise definition of Smale spaces is the following.

\begin{definition}[Smale space]
A \emph{Smale space} is a dynamical system $(X,f)$ consisting of a compact metrizable space $X$ equipped with an expansive  homeomorphism $f \colon X \to X$ which has the pseudo-orbit tracing property.
\end{definition}

The following result is known as~\emph{Anosov's closing lemma} (cf.~\cite[3.8]{bowen-equilibrium-2008}, \cite[Section~7.3]{ruelle-thermodynamic-2nd}).
It is a key ingredient in our proof of the first part of Theorem~\ref{t:main-surjunctive}.

\begin{theorem}[Anosov closing lemma for Smale spaces]
\label{t:nw-closure-per}
Let $(X,f)$ be a Smale space.
Then the set $\Per(X,f)$ is dense in  $\NW(X,f)$.
In other words, we have $\overline{\Per(X,f)} = \NW(X,f)$. 
\end{theorem}

Walters~\cite[Theorem~1]{walters-potp-1978} proved that a subshift $\Sigma \subset A^{\Z}$ has the pseudo-orbit tracing property if and only if $\Sigma$ is of finite type.
As full shifts are expansive and every subsystem of an expansive dynamical system is itself expansive,
it follows that a subshift  $\Sigma \subset A^{\Z}$ is a  Smale space if and only if $\Sigma$ is  of finite type.
Moreover, if $(X,f)$ is a  Smale space with $X$ totally disconnected, then there exist a finite set $A$ and a $1$-step subshift of finite type $\Sigma \subset A^{\Z}$ such that $(X,f)$ is topologically conjugate to $(\Sigma,\sigma)$. 

 \subsection{Coding of topologically mixing Smale spaces}
The theory of Markov partitions used by Rufus Bowen in the setting of basic sets of Axiom A diffeomorphisms directly extends to Smale spaces an yields 
in particular  the following result 
(see \cite[Theorem~7.6]{ruelle-thermodynamic-2nd},
\cite{bowen-markov-1970},
\cite{bowen-cbms-1978},
  \cite[Theorem~3.18 and Proposition~3.19]{bowen-equilibrium-2008}).
It is a key ingredient in our proof of Theorem~\ref{t:goe-for-irred-smale}.

\begin{theorem}
\label{t:coding-smale}
Let $(X,f)$ be a topologically mixing Smale space.
Then there exist a finite set $A$ and a topologically mixing subshift of finite type $\Sigma \subset A^{\Z}$
such that $(X,f)$ is a factor of $(\Sigma,\sigma)$.
\end{theorem}
\subsection{The Smale-Bowen-Ruelle decomposition theorem}

The following result is known as the \emph{Smale-Bowen-Ruelle spectral decomposition theorem} for non-wandering Smale spaces
(see~\cite[Section~3.B]{bowen-equilibrium-2008}, \cite[Section~7.4]{ruelle-thermodynamic-2nd},
\cite[Section~4.5]{putnam-notes-smale-spaces-2015}).

\begin{theorem}[The Smale-Bowen-Ruelle decomposition theorem]
\label{t:smale-decomposition}
Let $(X,f)$ be a non-wandering Smale space.
Then the following hold:
\begin{enumerate}[\rm (SBR1)]
\item 
the set
$\WW \coloneqq \{\overline{W^u(p)} : p \in \Per(X,f)\}$
is finite;
\item 
the elements of $\WW$ are non-empty clopen subsets of $X$ and form a finite partition of $X$;
\item 
one has $f(W) \in \WW$ for every $W \in \WW$;
\item
the map $\WW \to \WW$, $W \mapsto  f(W)$, 
is a permutation of $\WW$;
\item
if $W \in \WW$ and $k_W$ is the least positive integer such that $f^{k_W}(W)  = W$, then
the dynamical system $(W,f^{k_W})$ is a  topologically mixing Smale space; 
\item 
for every $W \in \WW$, the set
$\Omega_W \coloneqq \bigcup_{0 \leq i \leq k_W - 1} f^i(W)$ is a non-empty $f$-invariant clopen subset of $X$
and the dynamical system  $(\Omega_W,f)$ is an irreducible Smale space;
\item
if $Y$ is a non-empty $f$-invariant clopen subset of $X$ such that $(Y,f)$ is irreducible, then there exists $W \in \WW$ such that $Y = \Omega_W$.
\end{enumerate}   
\end{theorem}

\begin{lemma}
\label{l:tau-W-in-WW}
Let $(X,f)$ be a non-wandering Smale space and let $\WW$ be as in Theorem~\ref{t:smale-decomposition}.
Suppose that $\tau \in \End(X,f)$ and let $W \in \WW$.
Then there exists $W' \in \WW$ such that
$\tau(W) \subset W'$.
\end{lemma}

\begin{proof}
Let $p \in \Per(X,f)$ such that $W = \overline{W^u(p)}$.
We have $\tau(W^u(p)) \subset W^u(\tau(p))$
by Proposition~\ref{p:prop-morphism-stable}.(ii).
Using  the continuity of $\tau$,
we deduce that
$\tau(W) = \tau(\overline{W^u(p)}) \subset \overline{\tau(W^u(p))} \subset \overline{W^u(\tau(p))}$.
As $\tau(p) \in \Per(X,f)$ by Proposition~\ref{p:prop-morphism}.(ii), we deduce that we can take $W' \coloneqq \overline{W^u(\tau(p))} \in \WW$.  
\end{proof}

\subsection{The Li-Doucha Garden of Eden theorem}

The following result  was proved by Li~\cite{li-goe-spec-2019} and Doucha~\cite{doucha-goe-2023}, using properties of topological entropy,  
in the more general setting of  expansive continuous  actions with the weak specification property and the pseudo-orbit tracing property of countable amenable groups.
Actually,  the Myhill property for such systems was first obtained by Li~\cite[Theorem~1.1]{li-goe-spec-2019}
 and, a few years after, Doucha~\cite[Theorem~A]{doucha-goe-2023} established the Moore property.
 Note that the hypothesis that the system satisfies the pseudo-orbit tracing property is not needed for the Myhill part.   

\begin{theorem}
\label{t:li-doucha}
(Li~\cite{li-goe-spec-2019}, Doucha~\cite{doucha-goe-2023})
Let $(X,f)$ be a dynamical system consisting of a compact metrizable space $X$ and a homeomorphism $f \colon X \to X$.
Suppose that $(X,f)$ is expansive and has both the weak specification property and the pseudo-orbit tracing property.
Then $(X,f)$ satisfies the Garden of Eden theorem.
\end{theorem}

\section{Surjunctivity of non-wandering Smale spaces}
\label{sec:surjunctivity}

In this section, we prove the first part of Theorem~\ref{t:main-surjunctive}.
\par
The following auxiliary result  is a particular case of Proposition~5.1 in~\cite{csc-anosov-2016}.

\begin{lemma}
\label{l:dense-per-is-surj}
Let $X$ be a compact metrizable space and let $f \colon X \to X$ be a homeomorphism.
Suppose that the dynamical system $(X,f)$ is expansive and that the set $\Per(X,f)$ is dense in $X$.
Then $(X,f)$ is surjunctive.
\end{lemma}

\begin{proof}
Let $\tau \colon X \to X$ be an injective endomorphism of $(X,f)$.
Let $n \geq 1$ be an integer.
By Lemma~\ref{l:exp-finitely-p-per}, the set $\Per_n(X,f)$ is finite.
As  $\tau(\Per_n(X,f)) \subset \Per_n(X,f)$ by Proposition~\ref{p:prop-morphism}.(a),
injectivity of $\tau$ implies that $\tau(\Per_n(X,f)) = \Per_n(X,f)$.
Since $\Per(X,f) = \bigcup_{n \geq 1} \Per_n(X,f)$, it follows that $\tau(\Per(X,f)) = \Per(X,f)$.
The set  $\tau(X)$ is closed in $X$ by continuity of $\tau$ and compactness of $X$.
Since $\Per(X,f)$ is dense in $X$ by our hypotheses, we conclude that $\tau(X) = X$.
This shows that $\tau$ is surjective and hence that $(X,f)$ is surjunctive.
  \end{proof}

\begin{theorem}
\label{t:main-surjunctive-1}
Every non-wandering Smale space is surjunctive.
\end{theorem}

\begin{proof}
As every Smale space is expansive, this immediately follows from
Theorem~\ref{t:nw-closure-per} and Lemma~\ref{l:dense-per-is-surj}.
\end{proof}
 
\begin{remark}
\label{r:non-surjunctive-sft}
There exist non-surjunctive Smale spaces.
In fact, there are   subshifts of finite type that are not surjunctive.
The following example was given by Weiss~\cite[p.~358]{weiss-sgds} (cf.~\cite[Exercise~3.38]{csc-ecag}).
Consider the $1$-step subshift of finite type $\Sigma \subset \{0,1,2\}^{\Z}$ admitting the set
$\{00,11,22,01,12\}$ as a defining set of admissible words.
Observe that the word $12$ can appear at most once in a configuration of $\Sigma$.
It is easy to check that  the map $\tau \colon \Sigma \to \Sigma$, $x \mapsto \tau(x)$,  which replaces the word   $12$ (if it appears in $x$) by $11$,
is an  injective endomorphism of $\Sigma$ which is not surjective
(note that $012 \in L(\Sigma) \setminus L(\tau(\Sigma))$).
Therefore $(\Sigma,\sigma)$ is not surjunctive.  
\end{remark}

\begin{remark}
A surjective endomorphism of a Smale space $(X,f)$ may fail to be injective even if $(X,f)$ is topologically mixing.
For example, the endomorphism $\tau$ of the full shift $(\{0,1\}^{\Z},\sigma)$, defined by
$(\tau(u))_n \coloneqq u_n + u_{n +1} \mod 2$,
is easily shown to be surjective but not injective (cf.~\cite[Example~3.3.8]{csc-cag2}).
Another example is provided by \emph{Arnold's cat}, i.e., the Anosov diffeomorphism $f$ of the $2$-torus $X \coloneqq \R^2/\Z^2$
defined by $f(x_1,x_2) \coloneqq (2x_1 + x_2,x_1 + x_2)$,  and the endomorphism $\tau$ of $(X,f)$ given by $\tau(x_1,x_2) \coloneqq (2 x_1,2 x_2)$  (cf.~\cite[Section~4]{csc-anosov-2016}).   
\end{remark}

\begin{corollary}
\label{c:injective-endo-restricted-nw}
Let $(X,f)$ be a Smale space and let $\tau \in \End(X,f)$ be an injective endomorphism of $(X,f)$.
Then $\tau$ induces by restriction a dynamical system automorphism of $(\NW(X,f),f)$, i.e.,
an $f$-equivariant homemorphism of $\NW(X,f)$ onto itself.
\end{corollary}

\begin{proof}
We first observe that $\tau(\NW(X,f)) \subset \NW(X,f)$ by Proposition~\ref{p:prop-morphism}.(iii)
so that $\tau$ induces by restriction an endomorphism of $(\NW(X,f),f)$.
As $(\NW(X,f),f)$ is itself a Smale space (see \cite[p.~124]{ruelle-thermodynamic-2nd}) and $\tau$ is injective,
we deduce from Theorem~\ref{t:main-surjunctive-1} that $\tau$ induces by restriction a homeomorphism of $\NW(X,f)$.  
\end{proof}

\section{The Garden of Eden theorem for irreducible Smale spaces}
\label{sec:goe-irred-smale}

The goal of this section is to establish Theorem~\ref{t:goe-for-irred-smale}.
We shall use the following auxiliary results.

\begin{lemma}
\label{l:stably-equiv-same-compo}
Let $(X,f)$ be a non-wandering Smale space and let $\WW$ be as in Theorem~\ref{t:smale-decomposition}.
Suppose that $x,y \in X$ are stably equivalent (resp.~unstably equivalent, resp.~homoclinic)  and let $W$ be the unique element of $\WW$ such that $x \in W$.
Then one has $y \in W$.
\end{lemma}

\begin{proof}
Suppose that $x, y \in X$ are stably equivalent, i.e., $d(f^n(x),f^n(y)) \to 0$ as $n \to \infty$.
Let $k = k_W \geq 1$ be the integer as in Theorem~\ref{t:smale-decomposition}.(SBR5), so that $f^k(W) = W$.
Since $W$ is a compact open subset of $X$, there exists $\varepsilon > 0$ such that
$B_X(w,\varepsilon) \coloneqq \{w' \in X : d(w,w') < \varepsilon\} \subset W$ for every $w \in W$.
As $x$ and $y$ are stably equivalent, there exists $n \geq 1$ such that $d(f^{kn}(x),f^{kn}(y)) <  \varepsilon$.
Since $W$ is $f^{k n}$-invariant, 
this implies $f^{kn}(y) \in B_X(f^{kn}(x),\varepsilon) \subset W$ and hence $y \in W$.The proof in the case when $x$ and $y$ are unstably equivalent is similar.
As homoclinic points are stably  equivalent, the case
when $x$ and $y$ are homoclinic follows as well. 
 \end{proof}

\begin{lemma}
\label{l:goe-top-mixing-smale}
Every topologically mixing Smale space satisfies the Garden of Eden theorem.
\end{lemma}

\begin{proof}
Let $(X,f)$ be a topologically mixing Smale space.
By Theorem~\ref{t:coding-smale}, there exist a finite set $A$ and a topologically mixing subshift of finite type $\Sigma \subset A^{\Z}$ such that
$(X,f)$ is a factor of $(\Sigma,\sigma)$.
As $(\Sigma,\sigma)$ has the weak specification property by Proposition~\ref{p:tmftwsp},
it follows from  Proposition~\ref{p:prop-morphism}.(iv)
 that $(X,f)$ has also the weak specification property (cf.~\cite[Section~7.14]{ruelle-thermodynamic-2nd}).
   Since every Smale space is expansive and has the pseudo-orbit tracing property,
we deduce from Theorem~\ref{t:li-doucha}  that $(X,f)$ satisfies the Garden of Eden theorem.
\end{proof}

\begin{proof}[Proof of Theorem~\ref{t:goe-for-irred-smale}]
Let $(X,f)$ be an irreducible Smale space.
By the Smale-Bowen-Ruelle decomposition theorem (Theorem~\ref{t:smale-decomposition}), there exist $k \in \N$
and a finite partition $X = \bigsqcup_{i \in \Z/k\Z} W_i$ into non-empty clopen subsets
$W_i \subset X$, $i \in \Z/k\Z$,  such that $f(W_i) = W_{i + 1}$ and $(W_i,f^k)$ is topologically mixing for every $i \in \Z/k\Z$.
\par
Let $\tau \in \End(X,f)$.   
By Lemma~\ref{l:tau-W-in-WW}, there is a unique map
$\alpha \colon \Z/k\Z \to \Z/k\Z$ such that
$\tau(W_i) \subset W_{\alpha(i)}$ for all $i \in Z/k\Z$.
We have  that 
$\tau(W_{i + 1}) = \tau(f(W_i)) = f(\tau(W_i)) \subset f(W_{\alpha(i)}) = W_{\alpha(i) + 1}$ so that
$\alpha(i + 1) = \alpha(i) + 1$ for all $i \in \Z/k\Z$.
By induction, we get $\alpha(i) = i + \alpha(0)$ for all $i \in \Z/k\Z$.
This implies in particular that  $\alpha$ is a permutation of $\Z/k\Z$. 
\par
Let us choose, for each $i \in \Z/k\Z$, an integer $n_i \in \Z$ such that
$i - \alpha(i) = n_i  +  k\Z$.
Note that the map $\tau_i \colon W_i \to W_i$, defined by $\tau_i(x) \coloneqq  f^{n_i} ( \tau(x))$ for all $x \in W_i$,   is an  endomorphism of $(W_i,f^k)$.
\par 
Suppose first that $\tau$ is surjective.
 Then we must have $\tau(W_i) = W_{\alpha(i)}$ for all $i \in \Z/k\Z$.
Suppose that   $x,y \in X$ are homoclinic for $(X,f)$ with $\tau(x) = \tau(y)$.
Let $i$ be the unique element in $\Z/k\Z$ such that $x \in W_i$.
Since $x$ and $y$ are homoclinic,
it follows from Lemma~\ref{l:stably-equiv-same-compo}  that $y \in W_i$.
Consequently, we have   $\tau(x), \tau(y) \in W_{\alpha(i)}$.
Since the dynamical system $(W_i,f^k)$ is a topologically mixing Smale space, it has the Moore property by Lemma~\ref{l:goe-top-mixing-smale}.
As $\tau_i$ is a surjective endomorphism of $(W_i,f^k)$, we deduce that $\tau_i$  is   pre-injective.
Since $x$ and $y$ are homoclinic points of  $(W_i,f^{k})$ (by (i) $\implies$ (ii) in Proposition~\ref{p:homoclinic-f-f-k}) with the same image under $\tau_i$,
it follows that $x = y$.
This shows that $\tau$ is pre-injective and hence that $(X,f)$ has the Moore property.
\par
Suppose now that $\tau$ is a pre-injective endomorphism of $(X,f)$.
Let $i \in \Z/k\Z$.
If $x,y \in W_i$ are homoclinic for $(W_i,f^k)$ then they are homoclinic for $(X,f^k)$ and hence for $(X,f)$ 
(by using  (ii) $\implies$ (i) in Proposition~\ref{p:homoclinic-f-f-k}).
If, in addition, $\tau_i(x) = \tau_i(y)$ then $\tau(x) = \tau(y)$ so that $x = y$ by pre-injectivity of $\tau \in \End(X,f)$.
Consequently, the map $\tau_i  \colon W_i \to W_i$ is a pre-injective endomorphism of $(W_i,f^k)$.
As $(W_i,f^k)$ has the Myhill property by Lemma~\ref{l:goe-top-mixing-smale},
we deduce that $\tau_i(W_i) = W_i$.
This implies that  $\tau(W_i) = W_{\alpha(i)}$.
As $\alpha$ is a permutation of $\Z/k\Z$, we conclude that $\tau(X) = X$.
This shows that $\tau$ is surjective and hence that $(X,f)$ has the Moore property.
\end{proof}

\begin{remark}
\label{r:surj-equiv-preserve-measure}
There is also a purely measure-theoretic characterization of surjectivity for endomorphisms of irreducible Smale spaces.
Indeed, consider an irreducible Smale space $(X,f)$.
By the work of Bowen (see~\cite{bowen-markov-1970}, \cite{bowen-periodic-axiom-a-1971}, \cite{putnam-bowen-measures}), the dynamical system  $(X,f)$ is \emph{intrinsically ergodic} in the sense of~\cite{weiss-intrinsically-ergodic}.
This means that the topological entropy $h(X,f)$ of $(X,f)$ is finite
and there is a unique $f$-invariant Borel probability measure $\mu$ on $X$ whose measure-theoretic entropy $h_\mu(X,f)$ is equal to $h(X,f)$.
Moreover, the support of $\mu$ is the entire space $X$.
It then follows from Corollary~2.3 in~\cite{coven-paul} that an endomorphism $\tau$ of $(X,f)$ is surjective if and only if $\tau$ preserves $\mu$, i.e., if and only if $\mu(\tau^{-1}(B)) = \mu(B)$ for all Borel subsets $B \subset X$.
\end{remark}

\section{The Moore property for non-wandering Smale spaces}
\label{sec:nw-surj-moore}

In this section, we prove the second part of Theorem~\ref{t:main-surjunctive}.

\begin{theorem}
\label{t:moore-for-nw-smale}
Every non-wandering Smale space has the Moore property.
\end{theorem}

\begin{proof}
Let $(X,f)$ be a non-wandering Smale space and let $\tau \in \End(X,f)$ be a surjective endomorphism of $(X,f)$.
Let us show that $\tau$ is pre-injective. 
We keep the notation as in Theorem~\ref{t:smale-decomposition}.
It follows from Lemma~\ref{l:tau-W-in-WW} that for every $W \in \WW$ there exists a unique $\alpha(W) \in \WW$ such that $\tau(W) \subset \alpha(W)$. 
As $\tau$ is surjective, the map
$\alpha \colon \WW \to \WW$ is  a permutation of $\WW$ and  we have  $\tau(W) = \alpha(W)$ for every $W \in \WW$.
If $k$ denotes the order of the permutation $\alpha$ (i.e., the least integer $k \geq 1$ such that $\alpha^k = \Id_\WW$), we have $\tau^k(W) = \alpha^k(W) =  W$ for every $W \in \WW$.
Consequently, $\tau^k(\Omega_W) = \Omega_W$ for every $W \in \WW$. 
As the dynamical system $(\Omega_W,f)$ is an irreducible Smale space, 
it has the Moore property by Theorem~\ref{t:goe-for-irred-smale}. 
As $\tau^k\vert_{\Omega_W} \in \End(\Omega_W,f)$ is surjective, this implies that $\tau^k|_{\Omega_W}$ is  pre-injective.
Two homoclinic points in $X$ necessarily belong to the same $W \in \WW$
(cf.~Lemma~\ref{l:stably-equiv-same-compo}) 
and therefore to the same  $\Omega_W$.
 We deduce that $\tau^k \in \End(X,f)$ is pre-injective. 
It then follows from Corollary~\ref{c:pre-injective-powers} that $\tau$ itself is pre-injective. 
This shows that $(X,f)$ has the Moore property.
\end{proof}

\begin{remark}
\label{r:smale-space-not-myhill}
There exist non-wandering Smale spaces which do not have the Myhill property.
Consider for example the dynamical system $(X,f)$, where  $X = \{x_0,x_1\}$ is a discrete space with two distinct points and $f \coloneqq \Id_X$.
It is clear that $(X,f)$ is a non-wandering Smale space with two homoclinicity classes, namely, $\{x_0\}$ and $\{x_1\}$.
The map $\tau \colon X \to X$ defined by $\tau(x_0) = \tau(x_1) \coloneqq x_0$ is a pre-injective endomorphism of $(X,f)$ which is not surjective.
 Therefore, $(X,f)$ does not have the Myhill property.
 Observe that the dynamical system $(X,f)$ is topologically conjugate to the $1$-step subshift of finite type $\Sigma \subset \{0,1\}^{\Z}$
 consisting of the two constant configurations. 
\end{remark}

\begin{remark}
\label{r:smale-space-not-moore}
There exist Smale spaces which have neither the Moore property nor the Myhill property.
Indeed, consider  the Weiss subshift $\Sigma \subset \{0,1,2\}^{\Z}$ as in Remark~\ref{r:non-surjunctive-sft}.
We have already observed that $\Sigma$ is not surjunctive, so that, a fortiori, it does not have the Myhill property.
On the other hand, it is easy to see that the map $\tau \colon \Sigma \to \Sigma$, $x \mapsto \tau(x)$,  
which replaces the word $112$ (if it appears in $x$) by $122$,
is a surjective endomorphism of $\Sigma$ which is not pre-injective
(the sequences $u,v \in \Sigma$ defined by
$u_n = v_n \coloneqq 0$ for all $n \leq 0$, $u_1 = v_1 = v_2 \coloneqq 1$, $u_2 = 2$, and $u_n = v_n \coloneqq 2$ for all $n \geq 2$, are homoclinic and satisfy $\tau(u) = \tau(v) = u$).
Therefore $\Sigma$ does not have the Moore property either.  
\end{remark}

\bibliographystyle{siam}

\end{document}